\documentclass[11pt]{amsart}

\usepackage{amssymb,amsthm,amsmath}

\usepackage{mathtools}

\usepackage{hyperref}
\RequirePackage[dvipsnames,usenames]{xcolor}

\input{kmacros3.sty}

\usepackage{tikz}
\usepackage{tikz-cd}
\usetikzlibrary{cd}
\usepackage{graphicx}
\usepackage[all,cmtip]{xy}

\usepackage{mabliautoref}

\usepackage{bm}
\usepackage{pifont}
\usepackage{upgreek}

\usepackage{eucal}
\usepackage{ulem}
\usepackage{stmaryrd}

\numberwithin{equation}{theorem}

\usepackage{fullpage}

\usepackage{enumerate}
\usepackage{calc}

\usepackage{verbatim}
\usepackage{alltt,enumitem}
\usepackage{scalerel,stackengine}

\renewcommand{\m}{\mathfrak{m}}
\newcommand{\q}{\mathfrak{q}}

\title{\texorpdfstring{$F$}{F}-injectivity does not deform}
\author{Karl Schwede}
\address{Department of Mathematics, Unviersity of Utah, Salt Lake City, 84112 USA}
\email{schwede@math.utah.edu}
\urladdr{\url{https://kschwede.github.io/}}
\author{Austyn Simpson}
\address{Department of Mathematics, Bates College, 3 Andrews Rd, Lewiston, ME 04240 USA}
\email{asimpson2@bates.edu}
\urladdr{\url{https://austynsimpson.github.io/}}

\begin{document}
\begin{abstract}
    We show that there exists an $F$-finite four-dimensional local domain $(R,\mathfrak{m})$ of characteristic two which is not $F$-injective but which admits a nonzerodivisor $f\in \mathfrak{m}$ such that $R/fR$ is $F$-injective.  
\end{abstract}
\maketitle
\vspace*{-10pt}
\section{Introduction}
 A common strategy for showing that the singularities of a variety $X$ have certain mildness properties is to show that a (reduced) Cartier divisor $H$ on $X$ has those properties; there is a suite of so-called \emph{deformation} or \emph{inversion of adjunction} results one can then leverage to conclude that $X$ inherits those properties near $H$. In local commutative-algebraic language, this phenomenon is captured by the singularities of a Noetherian local ring $(R,\m)$ often being no worse than those of $R/fR$, where $f\in\m$ is a nonzerodivisor.

This philosophy has been extremely well-studied for the singularities of complex birational algebraic geometry, arguably beginning with Elkik's proof that rational singularities deform \cite{ElkikDeformationsOfRational}. Similar results hold for Du Bois singularities by \cite{KovacsSchwedeDBDeforms}, which are generalizations of both rational and log canonical singularities.

 This line of inquiry in the context of $F$-singularities in prime characteristic was initiated in 1983 by Fedder in his study of $F$-purity \cite{FedderFPureRat}, a notion which was then (mistakenly) conjectured to correspond to rational singularities in characteristic zero. Fedder proved in \emph{op. cit.} that $F$-purity need not deform in general, but that the closely related weaker notion of $F$-injectivity \emph{does} deform in a Cohen--Macaulay ring. Recall that a $d$-dimensional local ring $(R,\m)$ of characteristic $p > 0$ is \emph{$F$-injective} if the Frobenius actions $F:H^i_\m(R)\to H^i_\m(R)$ on local cohomology are injective for all $0\leq i\leq d$.

In the intervening decades, the deformation phenomena enjoyed by the Frobenius singularities have largely been found to mirror the picture of complex birational algebraic geometry. For instance, \emph{$F$-rational} singularities---which were subsequently found to be the correct counterparts of rational singularities by \cite{HaraRatImpliesFRat,MehtaSrinivasRatImpliesFRat, SmithFRatImpliesRat}---do deform by \cite{HochsterHunekeFRegularityTestElementsBaseChange}. On the other hand, log canonical and log terminal (respectively $F$-pure and $F$-regular) singularities do not deform in general, but do deform if the total space has a canonical divisor $K_X$ which is $\mathbb{Q}$-Cartier (see \cite{SinghFRegularityDoesNotDeform,KawakitaInversion,das_different_different_different,PST25}).

Beyond the Cohen--Macaulay case, the $F$-injective version of these theorems saw some partial progress but otherwise remained elusive. For example, if $R$ is excellent and $f\in\m$ is a nonzerodivisor such that $R/fR$ is $F$-injective, then $R$ is known to be $F$-injective provided that
\begin{enumerate}[label=(\roman*)]
    \item $R/fR$ is $F$-pure \cite[Corollary 4.13]{HoriuchiMillerShimomoto};\label{intro-1}
    \item $R/fR$ is either $F$-anti-nilpotent, or $F$-full \cite[Corollary 3.8]{MaQuyFrobeniusActionsAndDeformation};\label{intro-2}
    \item $H^i_\m(R)$ admits a Frobenius-stable secondary representation for all $0\leq i\leq d$ \cite{DM22};\label{intro-3}
    \item $\dim R\leq 4$, $R/fR$ is normal, and the residue field $R/\m$ is perfect \cite[Corollary 4.8]{HoriuchiMillerShimomoto};
    \item $R_f = R[f^{-1}]$ is Cohen-Macaulay \cite[Corollary A.5]{HoriuchiMillerShimomoto}.
\end{enumerate}
Most of these follow by showing that the relevant hypotheses ensure that the multiplication-by-$f$ maps
\begin{equation}
     H^i_\m(R)\stackrel{\cdot f} \to H^i_\m(R)\label{num.multByFSurjects}
\end{equation}
are surjective for all $0\leq i\leq\dim(R)$, at which point the diagram chases of \cite{FedderFPureRat,ElkikDeformationsOfRational} apply. Other work closely related to the deformation question for $F$-injectivity can be found, for instance, in \cite{HashimotoCMFinjectiveHoms,DattaMurayamaPermanencePropertiesFinjectivity,QS17,Mur22,DS24}. We show that the deformation question has a negative answer in general.
\begin{theorem}[{\autoref{cor.Main}, \autoref{prop.SpecialFiberIsFInjective}, \autoref{thm.AIsNotFInjective} }]\label{thm:main-thm}
    Let $k=\mathbb{F}_2(a)$. There exists a four-dimensional $F$-finite local domain $(R,\m)$, essentially of finite type over the DVR $D:=k[s]_{(s)}$, which is not $F$-injective but such that the quotient $R/sR$ by the nonzerodivisor $s$ is $F$-injective. In fact, $R$ is the localization of a graded finite-type $D$-algebra $A$. 
\end{theorem}
By taking a Veronese subalgebra, we may even assume that $A$ is standard graded over $D$, see \autoref{rem.SufficientlyHighMultipleIsStandardGraded}.

Due to the (partially conjectural) relationship between $F$-injective and Du Bois singularities via reduction to prime characteristic \cite{SchwedeFInjectiveAreDuBois, BhattSchwedeTakagiweakordinaryconjectureandFsingularity}, \autoref{thm:main-thm} gives one of the most prominent nodes of divergence between deformation of $F$-singularities and that of the singularities of complex birational geometry.  On the other hand, it is shown in \cite[Theorem 5.3]{MaSchwedeShimomoto} that if $(R, \fram)$ is a local ring essentially of finite type over $\C$, $f \in \fram$ is a nonzerodivisor, and $R/fR$ has dense $F$-injective type, then the same holds for $R$.  In fact, the proof of that result shows that the multiplication maps of \eqref{num.multByFSurjects} are surjective after reduction modulo $p \gg 0$ for a dense \emph{open} set of characteristic $p > 0$ models.  In this sense, the failure of deformation for $F$-injectivity is quite special.

In view of the affirmative results \ref{intro-1}-\ref{intro-3} as well as the implications
$$F\text{-pure}\Rightarrow F\text{-anti-nilpotent} \Rightarrow F\text{-full},$$
our quotient ring $A/sA$ must be manufactured to be $F$-injective but not $F$-full. The strategy we pursue here for producing such a ring mimics key parts of the approach of \cite{DSPS25}, which produces non-$F$-fullness by a purely inseparable base change mechanism on a Segre factor. Similar examples using various gluing procedures also appear in \cite{CuminoGrecoManaresiHyperplaneSections,EnescuHochsterTheFrobeniusStructureOfLocalCohomology,Enescu2009}.

One crucial departure from \cite{DSPS25} in the present paper is that in the ambient ring $A$ we work over the DVR $D:=k[s]_{(s)}$ rather than $k$, and treat $s$ as the deformation parameter. We produce $A$ by a gluing procedure via the domain $B:=D[\beta]/(\beta^2+s^2\beta +a)$, and the identification $B/sB\cong k[\sqrt{a}]$ is reminiscent of the well-studied loss of $F$-injectivity along a purely inseparable base change.

As an immediate corollary, we are able to resolve in the negative a question of De Stefani and Ma concerning the existence of Frobenius-stable secondary representations (see \cite[Question 4.1]{DM22} and \autoref{sec:corollaries}) using the same ring as in \autoref{thm:main-thm}. 

\begin{corollary}
    There exists a local $F$-finite domain $(R,\m)$ which has a local cohomology module $H^i_\m(R)$, none of whose secondary representations are $F$-stable.
\end{corollary}

\subsection*{Acknowledgements}

The authors thank Ben Baily, Linquan Ma, Anurag K. Singh, and Kevin Tucker for valuable conversations.  Karl Schwede was supported by NSF Grant \#2501903 as well as Simons Travel Support for Mathematicians SFI-MPSTSM-00013051.

\subsection*{AI Disclosure}

This example was found by OpenAI's ChatGPT 5.6 Sol Pro using a ChatGPT account provided to Karl Schwede by the University of Utah.  It was found first in a discussion where ChatGPT 5.6 Sol Pro had already identified some counter-examples to related problems (which will be written up in a separate work).  All writing in this article was done by the human authors.  Most, but not all, steps in the proof of the main result below follow the strategies ChatGPT 5.6 Sol Pro initially proposed.  The proof has also been substantially restructured to reflect the authors' understanding and tastes. ChatGPT 5.6 Sol Pro has also been used on this manuscript for the purpose of proofreading.  The human authors are responsible for this writeup. 

\section{Preliminaries} 

Throughout this article when working with schemes of characteristic $p > 0$, by Frobenius we always mean the \emph{absolute} Frobenius, even though our base fields are usually imperfect.

\subsection{A useful lemma}
We state the following global version of \cite[Proposition 3.4]{MaFinitenesspropertyoflocalcohomologyforFpurerings}.  

\begin{lemma}[{\cf \cite[Proposition 3.4]{MaFinitenesspropertyoflocalcohomologyforFpurerings}}]
    \label{lem.GlobalVersionOfMa}
    Suppose $k$ is a field of characteristic $p > 0$ and $Z$ is a globally $F$-split $k$-scheme with $H^0(Z, \cO_Z) = k$.  If $\sA$ is a line bundle on $Z$ and $a_1, \dots, a_r \in H^i(Z, \sA)$ are $k$-linearly independent, then their images under Frobenius,
    $a_1^p, \dots, a_r^p \in H^i(Z, \sA^p)$, are also $k$-linearly independent.  
\end{lemma}
\begin{proof}
    Suppose we have a $k$-linear relation on $a_1^p, \dots, a_r^p$, with $\sum_{i = 1}^r a_i^p \lambda_i  = 0$ for some $\lambda_i \in k$ at least one of which is nonzero, say $\lambda_s \neq 0$.  We can rewrite this relation as $\sum_{i=1}^r a_i F_* \lambda_i = 0 \in H^i(Z, F_* \sA^p)$.
    We fix $\phi : F_* \cO_Z \to \cO_Z$ to be the map obtained by pre-multiplying a splitting of Frobenius by $F_* \lambda_s^{-1}$ so that $\phi(F_* \lambda_s) = 1$ on all charts.  This gives us a map $\phi_k : F_* k = H^0(Z, F_* \cO_Z) \to H^0(Z, \cO_Z) = k$.  Tensor with $\sA$, and take $i$th cohomology to obtain
    \[
        \Phi : H^i(Z, F_* \sA^p) \cong H^i(Z, \sA \otimes F_* \cO_Z) \to H^i(Z, \sA \otimes \cO_Z) \cong H^i(Z, \sA).
    \]
    Applying $\Phi$ to our relation we obtain a new relation 
    \[
        0 = \sum_{i = 1}^r a_i \Phi(F_* \lambda_i) = \sum_{i=1}^r a_i \phi_k(F_* \lambda_i) \in H^i(Z, \sA).
    \]
    As $\phi_k(F_* \lambda_s) = 1$, this is a contradiction.
\end{proof}

\subsection{Ferrand's gluing construction}
\label{subsec.FerrandGluing}

In this paper we recall Ferrand's gluing construction and in particular the construction of line bundles on the associated pushouts that we will need.  First, we recall the following condition on a scheme:
\begin{center}
    (AF)  every finite collection of points is contained in an open affine subscheme.
\end{center}
This condition will follow immediately for us as the schemes we consider will be projective over a DVR (\cite[\href{https://stacks.math.columbia.edu/tag/01ZY}{Tag 01ZY}]{stacks-project}).

\begin{theorem}[{\cite[Theorem~5.4]{FerrandConducteurEtPincement}}]
\label{thm.FerrandGluing}
Suppose we have a pushout diagram in the category of ringed spaces:
\[
    \xymatrix{
        W \ar@{->>}[d]_{g} \ar@{^{(}->}[r]^e & Y \ar@{.>}[d]^{\nu}\\
        Z \ar@{.>}[r]_j & X
    }
\]
with $r = \nu \circ e = j \circ g$.  
Here we assume $W, Y, Z$ are schemes, $g$ is a finite surjective\footnote{surjective is not necessary, but it will be true in our case} map of schemes, and $e$ is a closed embedding.  Suppose additionally that $Y$ and $Z$ satisfy (AF).  Then $X$ is a scheme also satisfying (AF) and the diagram above is a diagram in the category of schemes.  Furthermore, $\nu$ is finite and $j$ is a closed embedding, and $\nu$ induces an isomorphism of $Y \setminus e(W)$ with $X \setminus j(Z)$.  Finally, the diagram is both a pushout (by construction) and a pullback.  In other words, it is both cocartesian (by construction) and cartesian.
\end{theorem}

We only need this in the following special case. For us
\[
W = W_1 \coprod W_2
\]
is a disjoint union of closed subschemes such that the induced maps $W_i \to Z$ are isomorphisms.  In this case, \emph{we are gluing $W_1$ to $W_2$ along the morphisms to $Z$.}

We label the following maps for future reference:
\[
g_i : W_i \hookrightarrow W \xrightarrow{g} Z \;\;\;\; \text{ and } \;\;\;\; e_i : W_i \hookrightarrow W \xrightarrow{e} Y
\]

We have corresponding surjective maps of sheaves:
\[
    \rho_i : \nu_* \cO_Y \to \nu_* (e_i)_* \cO_{W_i} = j_* (g_i)_* \cO_{W_i} \xrightarrow{\sim} j_* \cO_Z.
\]
With that notation, we have the following short exact sequence of sheaves on $X$:
\begin{equation}
    \label{eq.GluingSequenceGeneral}
    0 \to \cO_X \to \nu_* \cO_Y \xrightarrow{d_g} j_* \cO_Z \to 0
\end{equation}
where $d_g = \rho_1 - \rho_2$ is defined to be the difference of the $\rho_i$.  In particular, $d_g$ is not a map of sheaves of rings (it sends $1 \mapsto 0$).  We note $d_g$ is surjective by the Chinese Remainder Theorem as $W_1$ and $W_2$ are disjoint closed subschemes.  Since we will work exclusively in characteristic two, the difference of those maps is also the sum.  Finally, we note that Frobenius is compatible with \autoref{eq.GluingSequenceGeneral} in that we have a commutative diagram:
\begin{equation}
    \label{eq.GluingSequenceFrobeniusCompatible}
    \xymatrix{
        0 \ar[r] & \cO_X \ar[d] \ar[r] & \nu_* \cO_Y \ar[r]^{d_g} \ar[d] & j_* \cO_Z \ar[r] \ar[d] & 0 \\
        0 \ar[r] & F_* \cO_X \ar[r] & F_* \nu_* \cO_Y \ar[r]_{F_* d_g} & j_* F_* \cO_Z \ar[r] & 0 \\
    }
\end{equation}

In this setup, suppose we have line bundles $\sL_Y$ and $\sL_Z$ on $Y$ and $Z$ respectively and we have a fixed isomorphism 
\[
    \phi : e^* \sL_Y \xrightarrow{\sim} g^* \sL_Z.
\]
Then we define a sheaf on $X$ by the following formula.
\begin{equation}
    \label{eq.DefiningALineBundle}
    \sL_X = \ker \big( \nu_* \sL_Y \oplus j_* \sL_Z \xrightarrow{\Psi_{\phi}}  r_* g^* \sL_Z \big)
\end{equation}
where $\Psi_{\phi}$ is defined on a local section by $\Psi_{\phi}(v, m) = \phi(e^* v) - g^* m$. 
Thanks to \cite[Theorem~2.2(iv)]{FerrandConducteurEtPincement} we have that $\sL_X$ is finitely generated and locally free.  Further, by \cite[Theorem~2.2(i)]{FerrandConducteurEtPincement}, the following maps induced by \autoref{eq.DefiningALineBundle} are isomorphisms
\begin{equation}
    \label{eq.GluedLineBundlePullsBack}
\theta_Y : \nu^* \sL_X \xrightarrow{\sim} \sL_Y\;\;\;\; \text{ and } \;\;\;\; \theta_Z : j^* \sL_X \xrightarrow{\sim} \sL_Z.
\end{equation}
Indeed, these maps are defined via adjunction from the maps $\sL_X \to \nu_* \sL_Y$ and $\sL_X \to j_* \sL_Z$.  
In particular, since $\sL_Y$ is a line bundle and $\nu$ is surjective,  $\sL_X$ is also a line bundle.  
Even more, \cite[Theorem~2.2(i)]{FerrandConducteurEtPincement} guarantees that $\theta_Y$ and $\theta_Z$ are compatible with $\phi$ in the following sense:
\begin{equation}
    \label{eq.CompatibilitiesOfThetaMaps}
    \phi \circ (e^*  \theta_Y) = (g^* \theta_Z) \circ b
\end{equation}
where $b : e^* \nu^* \sL_X \xrightarrow{\sim} g^* j^* \sL_X$ is the isomorphism coming from the commutativity of the diagram.

We fix some additional notation.  Set $\rho_i^{\phi} : \nu_* \sL_Y \to j_* \sL_Z$ to be the map 
\[
    \rho_i^{\phi} : \nu_* \sL_Y \to \nu_* (e_i)_* e_i^* \sL_Y \xrightarrow{\nu_* (e_i)_* \phi|_{W_i} } j_* (g_i)_* g_i^* \sL_Z \xrightarrow{\sim} j_* \sL_Z.
\]

Analogous to \autoref{eq.GluingSequenceGeneral}, we have the following.

\begin{lemma}
    \label{lem.LineBundleSequence}
    There is a short exact sequence 
    \[
        0 \to \sL_X \to \nu_* \sL_Y \xrightarrow{\Delta_{\phi}} j_* \sL_Z \to 0
    \]
    obtained by tensoring the sequence \autoref{eq.GluingSequenceGeneral} by $\sL_X$.
    Here $\Delta_{\phi}$ is defined on a local section $y$ of $\sL_Y$ by
    \[
        y \mapsto \rho_1^{\phi}(y) - \rho_2^{\phi}(y).
    \]
\end{lemma}
\begin{proof}
    Note that $\phi = (g^* \theta_Z) \circ b \circ (e^* \theta_Y)^{-1}$ and in particular, $\theta_Y$ and $\theta_Z$ determine $\phi$.  
    Tensoring  \autoref{eq.GluingSequenceGeneral} by $\sL_X$ and using the projection formula gives us the following isomorphism of short exact sequences:
    \[
        \xymatrix@C=40pt{
            0 \ar[r] & \sL_X \ar@{=}[d] \ar[r] & \nu_* \nu^* \sL_X \ar[d]_{\nu_* \theta_Y}^{\sim} \ar[r]^{\delta} & j_* j^* \sL_X \ar[d]_{j_* \theta_Z}^{\sim} \ar[r] & 0\\
            0 \ar[r] & \sL_X \ar[r] &  \nu_* \sL_Y \ar[r]_-{\kappa} &  j_* \sL_Z \ar[r] & 0
        }
    \]
    where $\delta$ is induced by $\sL_X \otimes d_g$ and the projection formula and where $\kappa = (j_* \theta_Z) \circ \delta \circ (\nu_* \theta_Y)^{-1}$ makes the diagram commute.  The remainder of the proof is a careful computation in local charts which verifies that $\kappa = \Delta_{\phi}$.
    
    As this is a local computation, we may write $X = \Spec R$, $Y = \Spec S$, $Z = \Spec T$ and $W_i = \Spec S/{J_i}$.  We have canonical projection morphisms $\pi_i : S \to S/J_i$, $\pi : S \to S/J_1 \times S/J_2 = B$ and isomorphisms $\gamma_i : T \to S/{J_i}$ with map $\gamma : T \to S/J_1 \times S/J_2 = B$ defined by $\gamma(t) = (\gamma_1(t), \gamma_2(t))$.  Our sequence \autoref{eq.GluingSequenceGeneral} becomes
    \[
        0 \to R \to S \xrightarrow{\gamma_1^{-1} \circ \pi_1 - \gamma_2^{-1} \circ \pi_2} T \to 0
    \]
    where $d_g$ becomes $\gamma_1^{-1} \circ \pi_1 - \gamma_2^{-1} \circ \pi_2$,
    and we have ring maps $\zeta : R \to T$ and $n : R \to S$ such that $\gamma \circ \zeta = \pi \circ n$.
    Shrinking our open set if necessary, we let $L_S$ and $L_T$ be rank-1 free $S$- and $T$-modules respectively and a choice of isomorphism $\phi : L_S \otimes_S (S/J_1 \times S/J_2) \to L_T \otimes_T (S/J_1 \times S/J_2)$ which breaks down into isomorphisms $\phi_i : L_S \otimes_S S/J_i \to L_T \otimes_T S/J_i$.  Furthermore, we may also assume we have a rank-1 free module $L_R$ and isomorphisms  $\theta_S : L_R \otimes_R S \to L_S$ and $\theta_T : L_R \otimes_R T \to L_T$ such that 
    \[
        \phi = (\theta_T \otimes_T B) \circ b \circ (\theta_S^{-1} \otimes_S B) 
    \]
    where $b : L_R \otimes_R S \otimes_S B \to L_R \otimes_R T \otimes_T B$ is the canonical isomorphism.  Breaking this down further, on $S/J_i$ we have 
    \[
        \phi_i = (\theta_T \otimes_T S/J_i) \circ b_i \circ (\theta_S^{-1} \otimes_S S/J_i)
    \]
    where $b_i : L_R \otimes_R S \otimes_S S/J_i \to L_R \otimes_R T \otimes_T S/J_i$ is the canonical isomorphism.

    Via $\gamma_i$, we have isomorphisms $G_i : L_T \to L_T \otimes_{T} S/J_i$.  Pick $l \in L_R$ a free generator.  The local version of the map $\Delta_{\phi}$ is defined by it sending $s \theta_S(l \otimes 1_S) \in L_S$ for $s \in S$, to 
    \[
        G_1^{-1}(\phi_1(s \theta_S(l \otimes 1_S) \otimes 1_{S/J_1})) - G_2^{-1}(\phi_2(s \theta_S(l \otimes 1_S) \otimes 1_{S/J_2})).
    \]
    Unraveling our descriptions of $\phi_i$ in terms of the $\theta$ we compute 
    \[
        \begin{array}{rcl}
        G_i^{-1} (\phi_i(s \theta_S(l \otimes 1_S) \otimes 1_{S/J_i})) & = & G_i^{-1}( (\theta_T \otimes_T S/J_i) \Big( b_i \big( (\theta_S^{-1} \otimes_S S/J_i)(s\theta_S(l \otimes 1_S) \otimes 1_{S/J_i}) \big) \Big))\\
            & = & G_i^{-1} ((\theta_T \otimes_T S/J_i) \Big( b_i \big( l \otimes s \otimes 1_{S/J_i} \big) \Big))\\
            & = & G_i^{-1} ((\theta_T \otimes_T S/J_i)(l \otimes 1_T \otimes \pi_i(s)))\\
            & = & G_i^{-1} ((\theta_T(l \otimes 1_T)) \otimes \pi_i(s)) \\
            & = & \gamma_i^{-1}(\pi_i(s)) \theta_T(l \otimes 1_T)
        \end{array}
    \]
    In particular, $\Delta_{\phi}$ sends $s \theta_S(l \otimes 1_S)$ to $(\gamma_1^{-1}(\pi_1(s)) - \gamma_2^{-1}(\pi_2(s))) \theta_T(l \otimes 1_T)$.

    On the other hand, $\kappa$ on $s \theta_S(l \otimes 1_S)$ is computed as follows.  First we apply $\theta_S^{-1}$ to obtain $l \otimes s$.  Next we apply $\delta = {\id}_{L_R} \otimes_R d_g $ which sends $l \otimes s$ to $l \otimes (\gamma_1^{-1}(\pi_1(s)) - \gamma_2^{-1}(\pi_2(s)))$.  Finally, applying $\theta_T$ yields $\theta_T(l \otimes (\gamma_1^{-1}(\pi_1(s)) - \gamma_2^{-1}(\pi_2(s)))) = (\gamma_1^{-1}(\pi_1(s)) - \gamma_2^{-1}(\pi_2(s)))\theta_T(l \otimes 1)$.  Thus the two maps agree.
\end{proof}

\section{The construction}\label{sec:construction}

We sketch the construction.  We begin with the base schemes, then move to the non-normal 2-dimensional scheme $C$, the non-normal 3-dimensional scheme $X$, and finally the ring $A$ claimed in \autoref{thm:main-thm} which is a certain section ring over $X$.  

\subsection{The base schemes, extensions and variants}

We have the following rings.

\begin{enumerate}
    \item Set $k = \bF_2(a)$, an $F$-finite field.
    \item Set $D = k[s]_{(s)}$, a DVR.  Note $D/(s) = k$.    
    \item Set $B := D[\beta]/(\beta^2 + s^2 \beta + a)$.
    \item Set $K := B/sB = k[\beta]/(\beta^2 + a) \cong k[a^{1/2}]$.
\end{enumerate}

We see that $B \supseteq D$ is generically \'etale but the residue field extension is inseparable (it can be viewed as the Frobenius).  One quickly checks that 
\begin{equation}
    \label{eq.DefiningAutomorphismOnB}
    \iota : B \to B \;\;\;\;\; {\iota(\beta) = \beta + s^2}
\end{equation}
defines an automorphism of $B$ and also an automorphism of $D[\beta]$.

\subsection{The scheme $C$ and automorphism $\sigma$}

Write $Y = \bP^1_D = \Proj D[X_0, X_1]$ and set $Z = \Spec B$.  Consider the two homogeneous equations in $D[X_0,X_1]$.
\[
\begin{array}{rcl}
    F_P & = & X_0^2 + s^2 X_0 X_1 + a X_1^2\\
    F_Q & = & (X_0+X_1)^2 + s^2 (X_0+X_1) X_1 + a X_1^2
\end{array}
\]
and set $P = \bV_+(F_P)$ and $Q = \bV_+(F_Q)$.  A direct calculation shows that $P$ and $Q$ are both disjoint from $\bV_+(X_1)$.  
Furthermore, because $F_P + F_Q = X_1^2 + s^2 X_1^2 = (1 + s^2)X_1^2$ and $1+s^2$ is a unit, we see that $P$ and $Q$ are themselves disjoint.  Set $U = D_+(X_1)$.  Summarizing we have the following.

\begin{enumerate}
    \item $P \cap Q = \emptyset$.
    \item $P \cap \bV_{+}(X_1) = \emptyset$, in other words $P \subseteq U$.
    \item $Q \cap \bV_{+}(X_1) = \emptyset$, in other words $Q \subseteq U$.
\end{enumerate}

If we set $x = X_0/X_1$ to be the local variable on $U$, we see that the defining equations 
\[
    x^2 + s^2 x + a \;\;\;\text{ and }\;\;\; (x+1)^2 + s^2(x+1) + a
\]
both define subschemes isomorphic to $Z$ (for $P$, send $\beta \mapsto x$ and for the second send $\beta \mapsto x+1$).  Hence the subschemes $P, Q \subseteq Y$ are disjoint, affine and both are isomorphic to $Z$.  
We form the pushout diagram of schemes (Ferrand's \autoref{thm.FerrandGluing}) where $C$ is obtained by gluing the closed subschemes $P$ and $Q$ along their map to $Z$.
\begin{equation}
    \label{eq.PushoutSquareDefiningC}
\xymatrix{
    P \coprod Q \ar@{->>}[d]_{g} \ar@{^{(}->}[r]^-e & Y \ar@{->>}[d]^{\nu} \\
    Z \ar@{^{(}->}[r]_{j} & C
}
\end{equation}

We construct a certain involution $\sigma$ on $C$.  First consider the involution on $Y$ given by 
\[
    h : Y \to Y, \;\;\; [X_0 : X_1] \mapsto [X_0+s^2 X_1 : X_1].
\]
Note $h^2 = \id_Y$ as we are in characteristic $2$.
For the subscheme $P$, since
\[
    F_P \mapsto (X_0 + s^2 X_1)^2 + s^2 ((X_0+s^2 X_1 )X_1) + a X_1^2 = X_0^2 + 2s^4 X_1^2 + s^2 X_0 X_1 + aX_1^2 = F_P
\]
we see that $h(P) = P$.  A similar computation holds for $Q$ and so both $P$ and $Q$ are preserved by $h$.  Set $\tau : Z \to Z$ to be the map induced by $\iota$ from \autoref{eq.DefiningAutomorphismOnB}.  Under the identifications of $P$ and $Q$ with $Z$ above, we see $h$ induces the automorphism $\tau$  on both $P$ and $Q$.  Hence, we have a compatible automorphism on the diagram used to define $C$, \autoref{eq.PushoutSquareDefiningC}.  Therefore we have an automorphism 
\[  
    \sigma : C \to C
\] 
whose restriction to $j(Z)$ is $\tau$, which corresponds to $\iota$.  We note that $d_g \circ h^*$ is identified with $\iota \circ d_g$.  Summarizing all this, we have the following.
\begin{lemma}
    \label{lem.SigmaActionOnGluingSequence}
    With notation as above, $\sigma$ induces the following commutative diagram of short exact sequences of sheaves
    \[
    \xymatrix{
        0 \ar[r] &  \cO_C \ar[d] \ar[r] &  \nu_* \cO_{Y} \ar[d] \ar[r]^{d_g} & j_* \cO_Z \ar[d] \ar[r] &  0\\
        0 \ar[r] &  \sigma_* \cO_C \ar[r] &  \sigma_* \nu_* \cO_{Y} \ar[r]_-{\sigma_* d_g} & \sigma_* j_* \cO_Z \ar[r] & 0
    }
    \]
    where the right vertical map on $\cO_Z$ corresponds to $\iota$ and the middle vertical map is induced by $h$.
\end{lemma}
\begin{proof}
    This follows from the construction above.  
\end{proof}

\subsubsection{\texorpdfstring{The closed fiber $C_{s=0} \subseteq C$}{The closed fiber}}

Specializing \autoref{eq.GluingSequenceGeneral} to our situation, we have 
\[
    0 \to \cO_C \to \nu_* \cO_Y \xrightarrow{d_g} j_* \cO_Z \to 0.
\]
As every term is flat over $D$ (note $\cO_Y$ is flat since $Y = \bP^1_D$ and $\cO_C$ is torsion free as it is a subsheaf of $\cO_Y$), by base changing by $D/sD \cong k$, we obtain 
\begin{equation}
    \label{eq.GluingSequenceForCs=0}
    0 \to \cO_{C_{s=0}} \to \nu_* \cO_{\bP^1_k} \to j_* \cO_{Z_{s=0}} \to 0.
\end{equation}
This means that $C_{s=0} = C \times_D k$ is obtained by gluing two $B/sB = K$-points on $\bP^1_k$ together.  

\begin{remark}
    \label{rem.C0NotFSplit}
Although we will not need this, it is worth noting that $C_{s=0}$ is not globally $F$-split.  Indeed, if $C_{s = 0}$ were globally $F$-split, then the normalization $\bP^1_k$ would be compatibly split with the conductor (\cite[1.2.E, Exercise (4)]{BrionKumarFrobeniusSplitting}), which defines the subscheme $W_{s=0} = P_{s=0} \coprod Q_{s=0}$. But then by \cite[Theorem 1.2.8]{BrionKumarFrobeniusSplitting}, the map 
\[
    k^2 \cong H^0(\bP^1_k, \cO_{\bP^1_k}(1)) \to H^0(W_{s=0}, \cO_{W_{s=0}}) \cong K \oplus K
\]
is surjective, which is impossible.  
\end{remark}

\subsection{The scheme \texorpdfstring{$X$}{X}}

Now we form $Y' = C \times_D \bP^1_D = \bP^1_{C}$.  Note we have copies of $C$ at $0$ and $\infty$ given by 
\[
    e_0' : C^0 \to Y' \;\;\;\; \text{ and }\;\;\;\; e_{\infty}' : C^{\infty} \to Y'.
\]
We glue them together (again using Ferrand's \autoref{thm.FerrandGluing}), where the first is glued via the identity but the second is glued via $\sigma$.  
\begin{equation}
    \label{eq.SecondPushoutSquare}
    \xymatrix{
        C^0 \coprod C^{\infty} \ar@{^{(}->}[r]^-{e'} \ar[d]_{g' := \id_C \coprod \sigma}  & Y' \ar[d]^{\nu'} \\
        C \ar@{^{(}->}[r]_{j'} & X
    }
\end{equation}
Hence we identify $g'_0 : C^0 \to C$ as the identity and $g'_{\infty} : C^{\infty} \to C$ as $\sigma$.

A certain section ring of an ample line bundle on $X$ will be our non-$F$-injective ring.

\subsubsection{The closed fiber \texorpdfstring{$X_{s=0} \subseteq X$}{}}  
\label{subsubsec.ClosedFiberX0}
We specialize \autoref{eq.GluingSequenceGeneral} to our case and obtain:
\[
    0 \to \cO_X \to \nu'_* \cO_{Y'} \xrightarrow{d} j'_*\cO_C \to 0.
\]
We again base change by $D/sD \cong k$.  Similar to the above, all these sheaves are flat over $D$ and so we obtain a short exact sequence:
\[
    0 \to \cO_{X_{s=0}} \to \nu'_* \cO_{Y'_{s=0}} \xrightarrow{d} j'_* \cO_{C_{s=0}} \to 0.
\]
Hence $X_{s=0}$ is the gluing of $C^{\infty}_{s=0}$ to $C^{0}_{s=0}$ along $\id$ and $\sigma_{s=0}$.  But we notice that $\sigma_{s=0}$ is the identity as it is the identity on the $(s=0)$-base change of \autoref{eq.PushoutSquareDefiningC}.  It follows that $X_{s=0}$ is simply the scheme obtained by gluing $0$ and $\infty$ in $\bP^1_{C_{s=0}}$ to each other via the identity.  In particular, 
\[
    X_{s=0} \cong C_{s=0} \times_k N
\]
where $N$ is a nodal rational curve over $k$.  
This will be crucial in what follows.

\subsection{The line bundles on \texorpdfstring{$C$}{C} and on \texorpdfstring{$X$}{X}}

We first build an ample line bundle on $C$. Recall that $U = D_+(X_1) \subseteq Y = \bP_D^1$ is an affine chart that contains both $P$ and $Q$.  As the global section $X_1$ is nowhere vanishing on $U$, it gives an isomorphism $\cO_{P \coprod Q} \xrightarrow{\sim} e^* \cO_Y(1)$ induced by $1 \mapsto X_1$ (in other words, we view $\cO_Y(1)$ as the line bundle with poles of order at most $1$ along $V_+(X_1)$).  As $\cO_{P \coprod Q} \cong g^* \cO_Z$, we obtain an isomorphism $\phi_C : e^* \cO_Y(1) \to g^* \cO_Z$ (sending $X_1 \mapsto 1$).     

By \autoref{eq.DefiningALineBundle} and \autoref{eq.GluedLineBundlePullsBack}, this gives us a line bundle $\sN$ on $C$ such that  
\[
    \nu^* \sN = \cO_Y(1).
\]
Note as $\nu : Y \to C$ is finite surjective, $C$ is proper over $\Spec D$, and so we see that $\sN$ is ample \cite[Chapter I, Proposition 4.4]{Hartshorne.AmpleSubvarietiesBook}.  We keep track of $\sigma$ on this line bundle as well.
Again, by viewing $\cO_Y(1)$ as the sheaf with poles of order $1$ along $V_+(X_1)$, we see that $h : Y \to Y$ yields an isomorphism $h^* \cO_Y(1)$ with $\cO_Y(1)$ while simultaneously acting in a way compatible with $\tau$ on $P$ and $Q$.  Hence we also obtain an isomorphism:
\[
    \alpha : \sigma^* \sN \xrightarrow{\sim} \sN
\]
as well.  It is worth remarking that $\alpha \circ \sigma^* \alpha$ is the identity on $\sN$.

For any $t \geq 3$, we set 
\[
    \sM := \sN^{\otimes t}
\]
with corresponding $\alpha_{\sM} = \alpha^{\otimes t} : \sigma^* \sM \xrightarrow{\sim} \sM$.  

We will study the cohomology of this line bundle below, but for now we want to construct a line bundle on $X$.  First consider the line bundle on $Y' = C \times_D \bP^1_D$, 
\[
\sE = \sM \boxtimes \cO_{\bP^1_D}(t')
\]
for some $t' > 0$.  
Fix homogeneous coordinates $[T_0 : T_1]$ on $\bP^1_D$, the second factor of $Y'$, so that $T_0$ vanishes at $0$ (corresponding to $C^0$) and $T_1$ vanishes at $\infty$ (corresponding to $C^{\infty}$).
Via the canonical identifications $C^0 \cong C$, and identifying $T_1^{t'}$ with $1$, we obtain an isomorphism $e_0'^* \sE \cong \sM$.  Likewise, using the canonical $C^{\infty} \cong C$ and identifying $T_0^{t'}$ with $1$, we obtain an isomorphism $e_{\infty}'^* \sE \cong \sM$ (note there are other potential choices of isomorphism here, indeed these types of choices of isomorphism are why the picard group of the nodal rational curve is not simply $\bZ$).
On the other hand, via $g'$ we obtain:
\[
    g_1'^* \sM \cong \sM \;\;\;\;\text{ and } \;\;\;\; g_2'^*\sM \cong \sigma^* \sM.
\]
Using the identity on $C^0$ and the isomorphism $\alpha_{\sM} : \sM \xrightarrow{\sim} \sigma^* \sM$ on $C^{\infty}$, we have line bundles $\sM$ on $C$ and $\sE$ on $Y'$ as well as an isomorphism  
\begin{equation}
    \label{eq.PhiXDefinition}
    \phi_X : e'^* \sE \xrightarrow{\sim} g'^* \sM.
\end{equation}

Again, as in \autoref{subsec.FerrandGluing}, this gives us a line bundle $\sL$ on $X$ such that 
\[
    \nu'^* \sL \cong \sE \;\;\;\; \text{ and } \;\;\;\; j'^* \sL \cong \sM.
\]
Arguing just as before, since $\nu'$ is finite surjective and $\sE$ is ample, we also see that $\sL$ is ample (\cite[Chapter I, Proposition 4.4]{Hartshorne.AmpleSubvarietiesBook}).

\subsection{The section ring \texorpdfstring{$A$}{A}}
\label{subsec.TheSectionRing}
With notation as in the previous subsection, we define 
\begin{equation}
    A := \bigoplus_{n \geq 0} H^0(X, \sL^n)\label{eq:A-def}
\end{equation}
 
We shall prove that $A$ is not $F$-injective while $S = A/sA$ is $F$-injective.  As $A$ is graded over a DVR, after localizing at the homogeneous maximal ideal, this will prove the desired main result.

Although it will not be used in the sequel, we give an alternative description of the ring $A$ as a quotient of a polynomial ring over $D$ in \autoref{sec:presentation} for the values $t'=3$ and $t=6$.

\begin{remark}
    \label{rem.SufficientlyHighMultipleIsStandardGraded}
    By replacing $t$ and $t'$ with a sufficiently high multiple, we may assume that $A$ is generated in degree $A$ as a graded $R$-algebra over $A_0$.
\end{remark}

\section{Frobenius actions on cohomology of line bundles}

In this section we shall describe the Frobenius action on the cohomology of line bundles on $C, X$ and of the special fibers when $s = 0$.

\subsection{Cohomology on \texorpdfstring{$C$}{C}}

Our goal is to compute the cohomology of the structure sheaf $\cO_C$, and of $\sM^n$, on $C$.

\begin{lemma}
    \label{lem.CohomologyAndFrobeniusActionsOnOC}
    With notation as above, we have an isomorphism $H^0(C, \cO_C) = D$ and $H^1(C, \cO_C) \cong B$ which can be chosen compatibly with Frobenius.  Furthermore $H^1(C, \sM^n) = 0$ for all $n > 0$.  
\end{lemma}
\begin{proof}
By the construction of $C$ and \autoref{eq.GluingSequenceGeneral}, we have the following short exact sequence:
\begin{equation}
    \label{eq.PullbackSESforC}
    0 \to \cO_C \to \nu_* \cO_Y \xrightarrow{d_g} j_* \cO_Z \to 0
\end{equation}
where we recall that $d_g$ is given by $d_g(f) = \rho_1(f) + \rho_2(f)$ as described in \autoref{subsec.FerrandGluing} (the $\rho_i$ come from the isomorphisms $P, Q \cong Z$).  Explicitly, working on the affine chart $U = D_+(X_1)$ with $x = X_0/X_1$, then $\rho_1$ restricts to $\rho_P : D[x] \to B$, the $D$-algebra map defined by sending $x \mapsto \beta$; while $\rho_2$ restricts to $\rho_Q : D[x] \to B$, the $D$-algebra map defined by sending $x+1 \mapsto \beta$ (or equivalently, $x \mapsto \beta+1$).  In particular, $d_g$ is zero on $D$ while $d_g(x) = \beta + (\beta+1) = 1$.  

As $Z = \Spec B$, we see that $H^0(Z, \cO_Z) = B$.  Since $Y = \bP^1_D$ we have $H^1(Y, \cO_Y) = 0$ and $D = H^0(Y, \cO_Y)$ so that $H^0(C, d_g)$ is the zero map.  The long exact sequence on cohomology then becomes: 
\[
    0 \to H^0(C, \cO_C) \to D \xrightarrow{0} B \to H^1(C, \cO_C) \to 0.
\]
Thus we obtain the desired isomorphism.  The compatibility with Frobenius follows as Frobenius is compatible with \autoref{eq.PullbackSESforC} (as described in \autoref{eq.GluingSequenceFrobeniusCompatible}).

Twisting the sequence \autoref{eq.PullbackSESforC} by the line bundle $\sM^n$, and recalling that $\nu^* \sM^n \cong \cO_Y(t n)$, for some $t \geq 3$, and $j^* \sM^n \cong \cO_Z$ we obtain 
\begin{equation}
    \label{eq.TwistedSheafSequenceForC}
    0 \to \sM^n \to \nu_* \cO_Y(tn) \xrightarrow{d_n} j_* \cO_Z \to 0.
\end{equation}
We have a $D$-module basis of global sections $X_0^{tn}, X_0^{tn-1} X_1, \dots, X_1^{tn}$ for $H^0(Y, \cO_Y(tn))$.  We can see where $d_n$ sends these by restricting to $U$ (as $Z \subseteq U = \Spec D[x]$ with $x = X_0/X_1$).  In particular, as $t \geq 3$ and $n > 0$,
\[
d_n(X_0 X_1^{tn-1}) = d_n(X_0 X_1^{tn-1}|_U) = d_n(x) = \rho_P(x) + \rho_Q(x) = \beta + (\beta+1) = 1
\]
and likewise
\[
d_n(X_0^3 X_1^{tn - 3}) = d_n(x^3) = \beta^3 + (\beta+1)^3 = \beta^2 + \beta + 1 = (s^2 + 1)\beta + (1+a).
\]
Together, these two images generate $B$ as a $D$-module because $1+s^2 \in D$ is a unit.  Hence $H^0(Y, d_n)$ is surjective and we have the following long exact sequence, 
\[
    0 \to H^0(C, \sM^{n}) \to H^0(Y, \cO_Y(tn)) \twoheadrightarrow B \xrightarrow{0} H^1(C, \sM^n) \to H^1(Y, \cO_Y(tn))
\]
As $Y = \bP^1_D$ and $tn \geq 0$, we see this forces $H^1(C, \sM^n) = 0$ which completes the proof of the lemma.
\end{proof}

\subsection{Cohomology on \texorpdfstring{$X$}{X}}

Recall we have a line bundle $\sE = \sM \boxtimes \cO_{\bP^1_D}(t') = p^* \sM \otimes q^* \cO_{\bP^1_D}(t')$ on $Y' = C \times_D \bP^1_D$ where $p$ and $q$ are the projections.  We also have a line bundle $\sL$ on $X$ with $\sE = \nu'^* \sL$ and $\sM = j'^* \sL$.  

\begin{lemma}
    \label{lem.CohomologyOfX}
    With notation as above, $H^0(X, \cO_X) = D$ while $H^i(X, \sL^n) = 0$ for all $i > 0$ and $n > 0$.
\end{lemma}
\begin{proof}
    Fix an integer $n > 0$.  
    By the construction of $X$ and \autoref{eq.GluingSequenceGeneral}, we have a short exact sequence:
    \[
        0 \to \cO_{X} \to \nu'_* \cO_{Y'} \xrightarrow{d_{g'}} j'_* \cO_C \to 0
    \]
    As $Y' = C \times_D \bP^1_D = \bP^1_C$ and $H^0(C, \cO_C) = D$, and since $d_{g'}$ acts as zero on $D$ (as it is a difference of $D$-linear maps), we see that $H^0(X, \cO_X) = D$.  

    We twist by $\sL^n$ and obtain the following short exact sequence from \autoref{lem.LineBundleSequence}:
    \begin{equation}
        \label{eq.TwistedSequenceOfSheavesOnX}
        0 \to \sL^n \to \nu'_* (\sM^n \boxtimes \cO_{\bP^1_D}(t'n)) \xrightarrow{\Delta_{\phi_X}} j'_* \sM^n \to 0.
    \end{equation}
    where $\phi_X$ was defined immediately before \autoref{eq.PhiXDefinition} and $\Delta_{\phi_X}$ is defined in \autoref{lem.LineBundleSequence}.  Now, because $H^i(C, \sM^n) = 0$ for all $i > 0$ and $H^j(\bP^1_D, \cO_{\bP^1_D}(t'n)) = 0$ for all $j > 0$, and using that the zeroth cohomologies are $D$-flat, the K\"unneth formula gives 
    \[
        H^i(Y', \sM^n \boxtimes \cO_{\bP^1_D}(t'n)) = 0
    \]
    for all $i > 0$.  Therefore, to prove the lemma, by examining the long exact sequence on cohomology coming from \autoref{eq.TwistedSequenceOfSheavesOnX}, we see it suffices to prove that 
    \[
        H^0(X, \Delta_{\phi_X}) : H^0(Y', \sM^n \boxtimes \cO_{\bP^1_D}(t'n)) \to H^0(C, \sM^n)
    \]
    is surjective.

    Explicitly, $H^0(X, \Delta_{\phi_X})$ sends $u \in H^0\big( C \times_D \bP^1_D, \sM^n \boxtimes \cO_{\bP^1_D}(t'n)\big)$ to 
    \[
        \rho_1^{\phi_X}(u) + \rho_2^{\phi_X}(u)
    \]
    We write $[T_0 : T_1]$ for the homogeneous coordinates on the $\bP^1_D$-factor of $Y'$ so that for any $v \in H^0(C, \sM^n)$, $v \otimes T_1^{t'n}$ restricts to $v$ on $C^0 = \bV_+(T_0)$ and to $0$ on $C^{\infty} = \bV_+(T_1)$.  In particular, $\Delta_{\phi_X}(v \otimes T_1^{t'n}) = v$.  This finishes the proof of the lemma.
\end{proof}

\subsection{Frobenius action on cohomology}

We first write down the action $(\id + \iota) = (\id - \iota)$ on $B$ (the equality holds because we work in characteristic $2$).
For an arbitrary $c + d \beta \in B$ with $c,d \in D$, we have:
\[
    (\id + \iota)(c + d \beta) = (c + d \beta) + (c + d(\beta + s^2)) = d s^2
\]
We also record the action of the absolute Frobenius on $B$ in terms of this basis:
\[
    F_B(c + d\beta) = c^2 + d^2 \beta^2 = (c^2 + d^2 a) + d^2 s^2 \beta.
\]
\begin{lemma}
    \label{lem.FrobeniusCommutesWithIdPlusIota}
    With notation as above:  $(\id + \iota) \circ F_B = F_B \circ (\id + \iota)$
\end{lemma}
\begin{proof}
    Note we have 
    \[
        ((\id + \iota)\circ F_B)(c+d\beta) = (d^2 s^2) s^2 = d^2 s^4
    \]
    which agrees with 
    \[
        (F_B \circ (\id + \iota))(c + d \beta) = F_B(d s^2) = d^2 s^4.
    \]
\end{proof}

We view $B$ as a $D$-module with fixed $D$-basis $1, \beta$.
We then see that 
\begin{equation}
    \label{eq.KerAndCokerOfIdPlusIota}
    \ker(\id + \iota) = D \cdot 1 \oplus 0 \cdot \beta = D \subseteq B\;\;\;\; \text {and } \;\;\;\;\coker(\id +\iota) = (D/s^2 D) \cdot [1] \oplus D \cdot [\beta]
\end{equation} 
where $[1]$ and $[\beta]$ denote the image of $1$ and $\beta$ in the quotient.
By \autoref{lem.FrobeniusCommutesWithIdPlusIota}, as Frobenius commutes with $(\id + \iota)$, we have a Frobenius action on $\coker(\id +\iota)$ which can be described by its action on the basis elements as
\begin{equation}
    \label{eq.FrobeniusActionOnCokerIdPlusIota}
    F([1]) = [1] \;\;\; \text{ and } \;\;\; F([\beta]) = a [1] + s^2 [\beta].
\end{equation}
Finally, note that $F_K(K) = k \subsetneq K$ which we also can see from our original description $K \cong k(a^{1/2})$.

\begin{lemma}
    \label{lem.DescriptionAndFrobeniusActionOnCohomologyOfX}
    With notation as above, $H^1(X, \cO_X) \cong D^2$ and $H^2(X, \cO_X) \cong \coker(\id + \iota) = (D/s^2 D) \cdot [1] \oplus D \cdot [\beta]$.  Furthermore, the Frobenius action on $H^2(X, \cO_X)$ coincides via this isomorphism with the Frobenius action on $\coker(\id + \iota)$ described above in \autoref{eq.FrobeniusActionOnCokerIdPlusIota}.
\end{lemma}
\begin{proof}
    Again we begin with $0 \to \cO_{X} \to \nu'_* \cO_{Y'} \xrightarrow{d_{g'}} j'_* \cO_C \to 0$ and we consider the long exact sequence on cohomology:
    \[
        \xymatrix@R=10pt@C=40pt{
            & \dots  \ar[r] & H^0(Y', \cO_{Y'}) \ar[r]^{H^0(d_{g'})} & H^0(C, \cO_C) \\
            \ar[r] & H^1(X, \cO_X) \ar[r] & H^1(Y', \cO_{Y'}) \ar[r]^{H^1(d_{g'})} & H^1(C, \cO_C) \\
            \ar[r] & H^2(X, \cO_X) \ar[r] & H^2(Y', \cO_{Y'}) \ar[r] & 0.
        }
    \]
    As $H^0(Y', \cO_{Y'}) = D$, we see that $H^0(d_{g'})$ is zero.  As $Y' = C \times_D \bP^1_D$, the K\"unneth formula implies that $H^1(Y', \cO_{Y'}) = H^1(C, \cO_C) \cong B$ (see \autoref{lem.CohomologyAndFrobeniusActionsOnOC}) and $H^2(Y', \cO_{Y'}) = 0$.  Hence our sequence simplifies to a sequence of $D$-modules
    \[
    0 \to D \to H^1(X, \cO_X) \to B \xrightarrow{H^1(d_{g'})} B \to H^2(X ,\cO_X) \to 0
    \]
    compatible with Frobenius.

    \begin{claim}
        The map $H^1(d_{g'})$ is identified with $(\id + \iota)$.
    \end{claim}
    \begin{proof}[Proof of claim]
        Unraveling the definition of $H^1(d_{g'})$ we see it is the difference of two maps.  The first of these is the map $\gamma_1 : H^1(Y', \cO_{Y'}) \to H^1(C, \cO_C)$ induced by restricting a section to $C^0$ followed by the canonical isomorphism $C^0 \xrightarrow{\sim} C$.  The second of these is the map $\gamma_2$ induced by restricting to $C^{\infty}$, using the canonical isomorphism $C^{\infty} \xrightarrow{\sim} C$ followed by $\sigma : C \to C$.  
        
        By K\"unneth and \autoref{lem.CohomologyAndFrobeniusActionsOnOC}, $B \cong H^1(C, \cO_C) \otimes_D H^0(\bP^1_D, \cO_{\bP^1_D}) \cong H^1(Y', \cO_{Y'})$ and so we see $b \in H^1(Y', \cO_{Y'})$ is sent by $\gamma_1$ to the corresponding element of $B \cong H^1(C, \cO_C)$.  We also see that the involution $\sigma$ induces $\iota$ on $B \cong H^1(C, \cO_C)$ by \autoref{lem.SigmaActionOnGluingSequence}.  Hence $\gamma_2$ sends $b$ to $\iota(b)$.  This proves the claim.
    \end{proof}
    With the claim in place and applying \autoref{eq.KerAndCokerOfIdPlusIota}, we obtain a short sequence of $D$-modules $0 \to D \to H^1(X, \cO_X) \to D \to 0$ proving $H^1(X, \cO_X) \cong D^2$.  We also obtain an isomorphism $\coker(1 + \iota) = (D/s^2 D) \cdot [1] \oplus D \cdot [\beta] \cong H^2(X, \cO_X)$.  This and the compatibility with Frobenius we observed above prove the lemma.
\end{proof}

\subsection{Frobenius actions on cohomology of the closed fiber}
    \label{subsec.FrobeniusActionOnCohomologyClosedFiber}
    
First we need to describe the cohomology of the closed fiber.
Recall from \autoref{subsubsec.ClosedFiberX0} that  $X_{s=0} \cong C_{s=0} \times_k N$ where $N$ is a rational nodal curve.  We first compute cohomology on $C_{s=0}$.  
Taking cohomology of \autoref{eq.GluingSequenceForCs=0} gives the long exact sequence
\begin{equation}
    \label{eq.CohomologyCalculatorSequenceForCs=0}
    0 \to H^0(C_{s=0}, \cO_{C_{s=0}}) \to H^0(\bP^1_k, \cO_{\bP^1_k}) \xrightarrow{0=\overline{d_g}} B/sB \to H^1(C_{s=0}, \cO_{C_{s=0}}) \to 0
\end{equation}
and so 
\[
    H^0(C_{s=0}, \cO_{C_{s=0}}) = k \;\;\;\; \text{ and } \;\;\;\; H^1(C_{s=0}, \cO_{C_{s=0}}) \cong B/sB = K.
\]
Similarly, or classically, 
\[
    H^0(N, \cO_N) = k \;\;\;\; \text{ and } \;\;\;\; H^1(N, \cO_N) \cong k.
\]
The K\"unneth formula then gives that 
\begin{equation}
    \label{eq.CohomologyOfX0}
    H^0(X_{s=0}, \cO_{X_{s=0}}) = k, \;\;\;\; H^1(X_{s=0}, \cO_{X_{s=0}}) \cong k \oplus K  \;\;\;\; \text { and } \;\;\;\; H^2(X_{s=0}, \cO_{X_{s=0}}) \cong k \otimes K \cong K.
\end{equation}

\begin{lemma}
    \label{lem.FrobeniusInjectiveOnStructureSheafOfSpecialFiber}
    Frobenius acts injectively on $H^i(X_{s=0}, \cO_{X_{s=0}})$ for all $i \geq 0$.
\end{lemma}
\begin{proof}
    Since $X_{s=0}$ is reduced, the $i =0$ case is immediate.  First, we recall that Frobenius acts compatibly on \autoref{eq.CohomologyCalculatorSequenceForCs=0}, see \autoref{eq.GluingSequenceFrobeniusCompatible}.  Hence since Frobenius acts injectively on $B/sB = K$, it also acts injectively on $H^1(C_{s=0}, \cO_{C_{s=0}})$.  Frobenius acts injectively on $H^1(N, \cO_N)$ as $N$ is Frobenius split.  The $i = 1$ and $i = 2$ cases follow.  
\end{proof}

Now we handle the line bundle $\sL_{s=0} := \sL|_{s=0}$.  
For each $n > 0$ we have a short exact sequence $0 \to \sL^n \xrightarrow{\cdot s} \sL^n \to i_* \sL_{s=0}^n \to 0$.  By applying \autoref{lem.CohomologyOfX} we see that 
\begin{equation}
    \label{eq.HigherCohomologySpecialFiberIsZero}
    H^i(X_{s=0}, \sL_{s=0}^n) = 0 \;\;\;\; \text{ for all $i > 0$ and $n > 0$}.
\end{equation}

Next we prove a lemma that will be applied to both $C_{s=0}$ and to the nodal rational curve $N$ over $k$.  
\begin{lemma}
    \label{lem.FrobeniusInjectiveOnAllGluingsNegativeTwists}
    With notation as above, let $E = k$ or $E = K$.  Fix $P_E, Q_E \subseteq P^1_k$ disjoint closed points both of whose residue fields are isomorphic to $E$ and set $Z_E = \Spec E$.  Suppose the scheme $V$ is the pushout (as in \autoref{subsec.FerrandGluing}) of the diagram below
    \[
        \xymatrix{
            P_E \coprod Q_E \ar[d]_{g} \ar[r] & \bP^1_k \ar[d]^{\nu} \\
            Z_E \ar@{^{(}->}[r]_j & V
        }
    \]
    where the induced $P_E, Q_E \to Z_E$ are isomorpism identity.  We view $Z_E$ as a closed subscheme of $V$.  Suppose $\sQ$ is a line bundle on $V$ with $\nu^* \sQ \cong \cO_{\bP^1_k}(d)$ for some $d > 0$.  Then, for every $m > 0$, Frobenius is injective on $H^1(V, \sQ^{-m})$.
\end{lemma}
\begin{proof}
    As we saw in \autoref{eq.GluingSequenceFrobeniusCompatible}, Frobenius acts compatibly on $0 \to \cO_V \to \nu_* \cO_{\bP^1_k} \xrightarrow{d_g} j_* \cO_{Z_E} \to 0$ and twist that sequence by $\sQ^{-m}$.  After choosing an $E$-basis for $H^0(Z_E, j^* \sQ^{-m}) \cong E$ (recalling that $Z_E = \Spec E$), we have a short exact sequence
    \[
        0 \to E \to H^1(V, \sQ^{-m}) \to H^1(\bP^1_k, \cO_{\bP^1_k}(-dm)) \to 0.
    \]
    As Frobenius acts injectively on both the left and the right\footnote{since $\bP^1_k$ is globally Frobenius split}, the four lemma implies it acts injectively on the middle, proving the lemma.
\end{proof}

\begin{remark}
    If $V = N$, then it is well known that $N$ is Frobenius split\footnote{Note $\bP^1_k$ has a Frobenius splitting that compatibly splits the two $k$-rational points $0$ and $\infty$ (with the same induced Frobenius splittings of the residue fields at those points).  It is then not difficult to show that this splitting descends to a splitting on the nodal curve where $0$ and $\infty$ are glued together. }, and we also see directly that Frobenius is injective on $H^1(V, \sQ^{-m})$.  However, if $V = C_{s=0}$, then $V$ is not Frobenius split.  We saw this in \autoref{rem.C0NotFSplit} but it also follows from the fact that $K \cong H^1(C_{s=0}, \cO_{C_{s=0}})$, compatibly with Frobenius, and \autoref{lem.GlobalVersionOfMa} since the $k$-basis $1, \beta$ for $K \cong H^1(C_{s=0}, \cO_{C_{s=0}})$ are sent to $1, a$, which are linearly dependent over $k = \bF_2(a)$.
\end{remark}

For any integer $n$ we have a short exact sequence:
\[
    0 \to \sL^n \to \nu'_* \sM^n \boxtimes \cO_{\bP^1_D}(t'n) \xrightarrow{d_n} j'_* \sM^n \to 0
\]
which itself is obtained by twisting $0 \to \cO_X \to \nu'_* \cO_{Y'} \xrightarrow{d_{g'}} j'_* \cO_C \to 0$ by $\sL^n$, see \autoref{lem.LineBundleSequence}.  Again, as all these sheaves are torsion-free and so flat over $D$, we see that $\sL_{s=0}$ fits into 
\[
    0 \to \sL_{s=0}^n \to \nu'_* \sM_{s=0}^n \boxtimes \cO_{\bP^1_k}(t'n) \xrightarrow{\overline{d_n}} j'_* \sM_{s=0}^n \to 0.
\]
But on the closed fiber where the automorphism $\sigma$ is trivial, the map $\overline{d_{n}}$ is just the difference of the two sections under the natural identifications.  Using that $X_{s=0} = C_{s=0} \times_k N$ and our construction of $\sL$ (noting that $\sigma$ becomes trivial after setting $s = 0$) we then see that 
\[
    \sL_{s=0}^n = \sM_{s=0}^n \boxtimes \sP^{n}
\]
where $\sP$ is a line bundle on $N$ obtained by some choice of isomorphism $\cO_{\bP^1_k}(t')|_0 \cong \cO_{\bP^1_k}(t')|_{\infty}$. 
 We can be more precise about this gluing by keeping track of homogeneous coordinates, but what really matters is that $\sP$ is ample as it pulls back to an ample line bundle $\cO_{P^1_k}(t')$ on $\bP^1_k$.

\begin{lemma}
    \label{lem.FrobeniusActsInjectivelyNegativeTwistsSpecialFiber}
    Suppose $n < 0$.  Then $H^1(X_{s=0}, \sL_{s=0}^n) = 0$ and 
    \[
        H^2(X_{s=0}, \sL_{s=0}^n) = H^1(C_{s=0}, \sM_{s=0}^n) \otimes_k H^1(N, \sP^{n}).
    \]
    Furthermore, Frobenius acts injectively on $H^2(X_{s=0}, \sL_{s=0}^n)$.
\end{lemma}
\begin{proof}
    Since $H^0(C_{s=0}, \sM_{s=0}^n) = 0$ and $H^0(N, \sP^{n}) = 0$, we see from the K\"unneth formula the vanishing and the description of $H^2(X_{s=0}, \sL_{s=0}^n)$.  Thus we need to analyze the action of Frobenius on the latter.  Choose a $k$-basis $a_1, \dots, a_r \in H^1(N, \sP^{n})$ and consider an arbitrary nonzero element $\sum_{i = 1}^r c_i \otimes a_i \in H^1(C_{s=0}, \sM_{s=0}^n) \otimes_k H^1(N, \sP^{n})$.  If Frobenius sends this element to zero then 
    \[
        \sum_{i = 1}^r c_i^p \otimes a_i^p = 0.
    \]
    Since some $c_i \neq 0$, we also have that $c_i^p \neq 0$ by \autoref{lem.FrobeniusInjectiveOnAllGluingsNegativeTwists}.  Now, pick a $k$-linear projection $\pi : H^1(C_{s=0}, \sM_{s=0}^{pn}) \to k$ that sends $c_i^p$ to a nonzero element.  Applying this to our relation, we obtain 
    \[
        \sum_{i=1}^r \pi(c_i^p)  a_i^p = 0
    \]
    contradicting \autoref{lem.GlobalVersionOfMa}.
\end{proof}

\section{The main theorem}

The goal in this section is to put the pieces together.
Recall from \autoref{subsec.TheSectionRing} that $A = \bigoplus_{n \geq 0} H^0(X, \sL^n)$.  As $X$ is a $3$-dimensional integral scheme, $A$ is a $4$-dimensional integral domain.  Set $\fram = (s) + A_{> 0}$ to be the homogeneous maximal ideal.  
First we note that $A/sA$ is the section ring of $X_{s=0}$.

\begin{lemma}
    \label{lem.AmodSisSectionRing}
    We have an isomorphism of rings:
    \[
        A/sA \cong \bigoplus_{n \geq 0} H^0(X_{s=0}, \sL^n_{s=0})
    \]
    where $\sL^n_{s=0} = \sL^n|_{X_{s=0}}$.  
    As a consequence, $A/sA$ is an integral domain.
\end{lemma}
\begin{proof}    
    We begin with the short exact sequence $0 \to \cO_X \xrightarrow{\cdot s} \cO_X \to \cO_{X_{s=0}} \to 0$.  
    After twisting by $\sL^n$, it suffices to show that $H^0(X, \sL^n) \to H^0(X_{s=0}, \sL^n|_{s=0})$ surjects for all $n \geq 0$.
    We begin with $n = 0$.  Simply note that $H^0(X, \cO_X) = D$ by \autoref{lem.CohomologyOfX} and $H^0(X_{s=0}, \cO_{X_{s=0}}) = k \cong D/sD$ by \autoref{eq.CohomologyOfX0}.  For $n > 0$, $H^1(X, \sL^n) = 0$ by \autoref{lem.CohomologyOfX} and so the result follows.
\end{proof}

Now we can prove that $A/sA$ is $F$-injective.

\begin{proposition}
    \label{prop.SpecialFiberIsFInjective}
    $A/sA$ is $F$-injective.  
\end{proposition}
\begin{proof}
    By \autoref{lem.AmodSisSectionRing}, $A/sA$ is a section ring of an integral projective scheme $X_{s=0}$  over a field with respect to an ample line bundle.  As it is $\bN$-graded over an $F$-finite field, it suffices to check that $A/sA$ is $F$-injective at the irrelevant ideal $\frn = \fram / (s)$ (see for instance \cite[Chapter III, Proposition 5.11]{SchwedeSmith.FBook}).  As $X_{s=0}$ is an integral scheme of dimension 2, $A/sA$ is a domain of dimension 3.  Furthermore, as $A/sA$ is a section ring associated to an ample line bundle on a positive-dimensional projective variety over a field, it has depth $\geq 2$ at the irrelevant ideal (see for instance \cite[Proposition 2.1(1)]{HyrySmithCoreVersusGradedCore}).  It follows that $H^0_{\frn}(A/sA) = H^1_{\frn}(A/sA) = 0$.  Thus we must show that Frobenius injects on $H^2_{\frn}(A/sA)$ and $H^3_{\frn}(A/sA)$.  By the connection between local cohomology and sheaf cohomology, for $i \geq 2$, 
    \[
        [H^i_{\frn}(A/sA)]_n = H^{i-1}(X_{s=0}, \sL^n_{s=0})
    \]
    and this isomorphism is compatible with Frobenius.  Thus we must show that Frobenius acts injectively on $H^{1}(X_{s=0}, \sL^n_{s=0})$ and on $H^2(X_{s=0}, \sL^n_{s=0})$ for all $n \in \bZ$.  
    
    For $n > 0$, these cohomologies are zero by \autoref{eq.HigherCohomologySpecialFiberIsZero}.  For $n = 0$, Frobenius acts injectively by \autoref{lem.FrobeniusInjectiveOnStructureSheafOfSpecialFiber}.  For $n < 0$, Frobenius acts injectively by \autoref{lem.FrobeniusActsInjectivelyNegativeTwistsSpecialFiber}.  The result is proven.
\end{proof}

\begin{theorem}
    \label{thm.AIsNotFInjective}
    $A$ is not $F$-injective.  In fact, Frobenius annihilates a class in $H^4_{\fram}(A)$.  
\end{theorem}
\begin{proof}
    We shall construct a class in $H^4_{\fram}(A)$ that is annihilated by Frobenius.  First note that we have an isomorphism
    \[
        [H^3_{A_{>0}}(A)]_0 \cong H^2(X, \cO_X) \cong \coker(\id + \iota) = ((D/s^2 D) \cdot [1]) \bigoplus (D \cdot [\beta])
    \]
    by \autoref{lem.DescriptionAndFrobeniusActionOnCohomologyOfX} which is compatible with Frobenius ($F([1]) = [1]$ and $F([\beta]) = a[1] + s^2 [\beta]$ as explained in \autoref{eq.FrobeniusActionOnCokerIdPlusIota}).  

    From $\myR\Gamma_{(s)} \circ \myR\Gamma_{A_{>0}} \cong \myR\Gamma_{\fram}$ we have the spectral sequence
    \[
        E_2^{u,v} = H^u_{(s)}(H^v_{A_{>0}}(A)) \Rightarrow H^{u+v}_{\fram}(A).
    \]
    For $u+v = 4$, the only possible nonzero terms are $E_2^{0,4} = H^0_{(s)}(H^4_{A_{>0}}(A))$ and $E_2^{1,3} = H^1_{(s)}(H^3_{A_{>0}}(A))$.  Furthermore the differentials into and out of those terms are zero on the $E_2$ and every subsequent page (there are only two nonzero columns).  From the induced short exact sequence $0 \to H^1_{(s)}(H^3_{A_{>0}}(A)) \to H^4_{\fram}(A) \to H^0_{(s)}(H^4_{\fram}(A)) \to 0$ we see it suffices to prove that Frobenius is not injective on 
    \[
        [H^1_{(s)}\big(H^3_{A_{>0}}(A)\big)]_0 \cong H^1_{(s)}\big((D/s^2 D) \cdot [1] \oplus D \cdot [\beta] \big).
    \]
    The right side is the cokernel of the map
    \begin{equation}
        \label{eq.DescriptionOfH4mAsCokernel}
        ((D/s^2 D) \cdot [1]) \bigoplus (D \cdot [\beta]) \to \Big( ((D/s^2 D) \cdot [1]) \bigoplus (D \cdot [\beta]) \Big)_s \cong (0 \cdot[1])_s \bigoplus (D\cdot[\beta])_s
    \end{equation}
    where $(-)_s$ denotes localization at $s$.  Thus consider the nonzero class in the cokernel:
    \begin{equation}
        \left[ \frac{0 \oplus [\beta]}{s} \right]\label{eq:Cech-class-in-kernel}
    \end{equation}
        
    Frobenius sends this class to 
    \[
        \left[\frac{\overline{a}[1] \oplus s^2 [\beta]} {s^2} \right] = \left[\frac{0[1] \oplus s^2 [\beta]}{s^2}\right] = \left[\frac{0[1] \oplus [\beta]}{1}\right].
    \]
    But that comes from the image of \autoref{eq.DescriptionOfH4mAsCokernel} and hence it is zero.  This completes the proof.    
\end{proof}

As a consequence, we immediately obtain the local version as well.
\begin{corollary}
    \label{cor.Main}
    Set $R = A_{\fram}$ so that $R$ is an $F$-finite 4-dimensional domain essentially of finite type over $D = \bF_2(a)[s]_{(s)}$.  Then $R/sR$ is $F$-injective while $R$ is not.
\end{corollary}
\begin{proof}
    There are natural isomorphisms $H^i_{\fram R}(R) \cong H^i_{\fram}(A)$ and $H^i_{\fram R/(s R)}(R/sR) \cong H^i_{\fram/(s)}(A/sA)$ for all $i$, see for instance \cite[Corollary 4.3.3]{BrodmannSharpLocalCohomology} (or one can deduce it from \cite[\href{https://stacks.math.columbia.edu/tag/0ALZ}{Tag 0ALZ}]{stacks-project}) noting that $H^i_{\fram}(A)$ and $H^i_{\fram/(s)}(A/sA)$ are supported at the homogeneous maximal ideal.  
\end{proof}

\section{Corollaries}\label{sec:corollaries}
We now turn to two related questions that our example resolves. The first concerns Frobenius-stable secondary representations of the local cohomology modules of a local ring. We recall the following definition from \cite[Section 7.2]{BrodmannSharpLocalCohomology} and \cite{DM22}:
\begin{definition}
    Let $(R,\m,k)$ be a $d$-dimensional local ring of prime characteristic $p>0$, and let $M$ be an $R$-module.
    \begin{enumerate}[label=(\roman*)]
        \item Say that $M$ is \emph{secondary} if $M\neq 0$ and for every element $x\in R$, the multiplication-by-$x$ map $M\to M$ is either surjective or nilpotent.
        \item A \emph{secondary representation} of $M$ is a decomposition $M=\sum\limits_{j=1}^n M_j$ where each $M_j\subseteq M$ is a secondary submodule.
        \item For each $0\leq i\leq d$, we say that a secondary representation of the $i$-{th} local cohomology module $H^i_\m(R)=\sum\limits_{j=1}^n M_j$ is \emph{$F$-stable} if $F(M_j)\subseteq M_j$ for each $j$, where $F$ is the natural Frobenius action on $H^i_\m(R)$.
    \end{enumerate}
\end{definition}
It is known that $H^{\dim(R)}_\m(R)$ \emph{always} admits an $F$-stable secondary representation, but it was open whether this holds for the lower local cohomology modules as well; see \cite[Lemma 3.2 and Question 4.1]{DM22}. The motivation is that \cite[Theorem 3.4]{DM22} proves that if each local cohomology module $H^i_\m(R)$ admits a $F$-stable secondary representation and if $R/fR$ is $F$-injective for some nonzerodivisor $f\in \m$, then $R$ is $F$-injective. It thus follows immediately from \autoref{thm:main-thm} that $F$-stable secondary representations need not exist.
\begin{corollary}
    There exists a local $F$-finite domain $(R,\m)$ which has a local cohomology module $H^i_\m(R)$, none of whose secondary representations are $F$-stable.
\end{corollary}

Another perspective that has been studied in relation to the deformation problem is that of $\m$-adic stability, which asks whether, given a local ring $(R,\m)$ and a nonzerodivisor $f\in\m$ such that $R/fR$ satisfies some property $\mathcal{P}$, there exists an integer $N\gg 0$ such that $R/(f+\epsilon)R$ also satisfies $\mathcal{P}$ for all $\epsilon\in \m^N$. This has been pursued in \cite{DM22,PS23} in the context of $F$-singularities, and it is proven in \cite[Theorem 2.4]{DS24} that $\m$-adic stability implies deformation for many local properties arising from geometric settings. Our counterexample to $F$-injective deformation thus also yields:

\begin{corollary}
    There exists a five-dimensional $F$-finite local domain $(R,\m)$ and a nonzerodivisor $f\in \m$ such that $R/fR$ is $F$-injective, but such that for every integer $N\geq 1$ there exists an element $\epsilon\in \m^N$ such that the ring $R/(f+\epsilon)R$ is not $F$-injective.
\end{corollary}
\begin{proof}
    Since $F$-injectivity descends under flat local ring homomorphisms and persists after adjoining a variable and localizing (see \cite{DattaMurayamaPermanencePropertiesFinjectivity}), both conditions of \cite[Theorem 2.4]{DS24} are satisfied. The result then follows from \autoref{thm:main-thm} and the construction in \cite[Proof of Theorem 2.4]{DS24}.
\end{proof}

\section{An equational description of \texorpdfstring{$A$}{A}}\label{sec:presentation}
In this section we sketch how to realize the ring $A$ as an explicit quotient of a polynomial ring over $k[s]_{(s)}$. We specialize to the case of $t'=3$ and $t=6$ in the notation of \autoref{sec:construction}. Because we do not see at present a way to give a self-contained equational proof without relying on the results proved in earlier sections, we have elected to omit proofs of several statements, most of which can otherwise be verified in \texttt{Macaulay2} \cite{M2}. Recall the following notation.
\begin{enumerate}
    \item Set $k = \bF_2(a)$ and $D = k[s]_{(s)}$.
    \item In the DVR $D$, set $c=1+s^2$, $b=s^2c$, and $d=a^2+ac$. In $D[x]$, set $y=x^2+x$.
\end{enumerate}
 
 Consider the map $\psi:D[e_0,e_1,e_2,e_3,e_4]\to D[x,u]$ by sending 
\begin{align}
\psi:\begin{cases}
    e_0\mapsto u\\
    e_1\mapsto yu\\
    e_2\mapsto (y^2+by)u\\
    e_3\mapsto (y^3+by^2)u\\
    e_4\mapsto (x^2+s^2 x)(y^2+by+d)u.
\end{cases}\label{eq:psi}
\end{align}
Then let $\mathfrak{c}=\ker(\psi)$, and set 
\[
C:=\frac{D[e_0,e_1,e_2,e_3,e_4]}{\mathfrak{c}}.
\]
We identify $C$ with the subalgebra $D[u,yu,y^2u,y^3 u,x(y^2+by+d)u]\subseteq D[x,u]$ and note that the change of basis $u,yu, y^2 u, y^3z,x(y^2+by+d)u$ to the images in \eqref{eq:psi} has determinant $c\in D^*$. One may then verify that $\mathfrak{c}$ is generated by the size two minors of the matrices \begin{align*}
    \begin{pmatrix}
        e_0&e_1&e_2\\
        e_1+be_0&e_2&e_3+be_2
    \end{pmatrix},\quad \begin{pmatrix}
        e_4+c(e_2+de_0)&e_3+de_1\\
        e_3+de_1+b(e_2+de_0)&e_4
    \end{pmatrix}.
\end{align*} 

We can now describe a presentation for the ring $A$. Under the automorphism $x\mapsto x+s^2$, note that $y\mapsto y+b$. One checks that this induces
\begin{align*}
    \begin{cases}
        e_0\mapsto e_0\\
        e_1\mapsto e_1+be_0\\
        e_2\mapsto e_2\\
        e_3\mapsto e_3+be_2\\
        e_4\mapsto e_4
    \end{cases}
\end{align*}
so in particular this assignment fixes $\mathfrak{c}$. Let $\varphi:C\to C$ be the induced ring automorphism and consider the map 
\begin{align*}
    \rho:D[x_i,y_i,z_i\mid 0\leq i\leq 4] \to C[T_0,T_1],\quad
    \begin{cases}
        x_i\mapsto e_i T_0 T_1^2\\
        y_i\mapsto e_i T_0^2 T_1\\
        z_i\mapsto e_i T_1^3+\varphi(e_i)T_0^3.
    \end{cases}
\end{align*}

\begin{proposition}
    Let $\mathfrak{a}:=\ker(\rho)$.  The ring $A:= D[x_i,y_i,z_i\mid 0\leq i\leq 4]/\mathfrak{a}$ coincides with the section ring constructed in \autoref{sec:construction} for the values $t'=3$ and $t=6$.
\end{proposition}
\begin{proof}
 Indeed, first note that $C=\bigoplus\limits_{n\geq 0}C_n$ is generated by $C_1$. Consider now a term of the form $h T_0^j T_1^{3n-j}$ where $0<j<3n$. We will show by induction on $n$ that this element belongs to $\im(\rho)$. This is clear if $n\leq j\leq 2n$. For \(2n<j<3n\),  write
\[
h=\sum_\ell\varphi(a_\ell)h_\ell,
\qquad
a_\ell\in C_1,\quad h_\ell\in C_{n-1}.
\]
By induction on $n$, observe that $h_\ell T_0^{j-3}T_1^{3n-j}\in\operatorname{im}(\rho).$ Multiplication by
\[
a_\ell T_1^3+\varphi(a_\ell)T_0^3
\]
produces the desired term with $T_0$-exponent $j$, together with a term of
exponent $j-3$.  Beginning with \(j=2n+1\) and proceeding upward, the latter
term lies either in the range \(n\leq j-3\leq2n\) or is already in $\im(\rho)$. The $j<n$ case is handled similarly by a downward induction.

For an arbitrary $h\in C_n$, write $h$ as a sum of elements of the form $a_1\cdots a_n$ where each $a_i\in C_1$. Then the outer terms of the product
\begin{equation}
    \prod\limits_{i=1}^n (a_iT_1^3+\varphi(a_i)T_0^3)
\end{equation}
are $hT_1^{3n}+\varphi(h)T_0^{3n}$ and the remaining terms have $T_0$-exponent strictly between $0$ and $3n$.
\end{proof}

Now observe that the quotient ring is the Segre product $$A/sA\cong (C/sC) \# \frac{k[v_1,v_2,v_3]}{(v_1^3+v_2^3+v_1v_2v_3)}$$ under the identification $v_1=T_0 T_1^2$, $v_2=T_0^2 T_1$, and $v_3=T_0^3+T_1^3$.

We now describe a \v{C}ech class $\xi\in H^4_{\m}(A)$ in the kernel of the Frobenius action which is inspired by the identifications above with $c^2$ times the \v{C}ech class in \eqref{eq:Cech-class-in-kernel}. The elements $s,x_0, z_4,x_4+z_0$ may be checked in \texttt{Macaulay2} to be a system of parameters for $A$. 

We compute local cohomology of $A$ via the \v{C}ech complex on these elements, letting $\mathfrak{q}$ denote the ideal that they generate. Denote the product without $s$ by $f=x_0z_4(x_4+z_0)$. To streamline the notation, for a $5$-tuple $\uline{g}$ consider the following operators.
\begin{align*}
    \Phi(\uline{g}):=&g_2g_3+g_2g_4+ag_1g_3+(a+c^2)g_1g_4+adg_1^2+a(a+c^2)g_0g_3+acg_0g_4+abdg_0g_1\\
    \Psi(\uline{g}):=& g_1g_3+g_1g_4+cg_0g_4+bg_0g_3+dg_1^2+bdg_0g_1\\
    \Lambda(\uline{g}):=&cg_4+s^2 cg_3+bc^2g_2+c(b^2c+a^2c+d)g_1\\
    \Omega(\uline{g}):=&s^2 c^2 \Phi(\uline{g})+ac^3\Psi(\uline{g})
\end{align*}

Consider the homogeneous element of degree three
\begin{align*}
    \alpha&=y_0\Phi(\uline{x})+z_0\Phi(\uline{y})\in A
\end{align*}
and the \v{C}ech class $\xi = \left[\frac{\alpha}{sx_0z_4(x_4+z_0)}\right]\in\left[H^4_\m(A)\right]_0$.  To verify that $A$ is not $F$-injective from this perspective, it would be enough to show that $F(\xi)=0$ and that $\xi\neq 0$. A method is sketched below to confirm that $F(\xi)=0$, which amounts to finding an integer $r\geq 2$ such that $\alpha^2 s^{r-2}f^{r-2}\in \q_r:=(s^r,x_0^r,z_4^r,(x_4+z_0)^r)$. However, checking that $\xi\neq 0$ purely algebraically would require showing that $\alpha s^{r-1}f^{r-1}\not\in\q_r$ for \emph{every} $r\geq 1$, but we have not verified this.

We isolate the following useful identities which are verifiable in \texttt{Macaulay2}.
\begin{lemma}\label{lem:poly-identity} In the polynomial ring $D[x_i,y_i,z_i\mid 0\le i\le4]$,
we have:
\begin{enumerate}[label=(\roman*)]
    \item\label{lemma:equiv-1} the following congruence holds
    \begin{equation}
        \alpha^2+s^2 c^2\alpha f+ac^3(y_0\Psi(\uline{x})+z_0\Psi(\uline{y}))f\equiv x_0^2R_1+z_4^2R_2\bmod \mathfrak{a}
    \end{equation}
    where
    \begin{align*}
        R_1 = &\Phi(\uline{x})\Phi(\uline{y})+x_4(x_4\Omega(\uline{x})+y_4\Omega(\uline{y}))+z_4^2(\Omega(\uline{y})+y_4\Lambda(\uline{y})+y_0\Lambda(\uline{x}))\\
        R_2=&x_0(x_4+z_0)y_0\Lambda(\uline{y}).
    \end{align*}
    \item\label{lemma:equiv-2}$ac^3(y_0\Psi(\uline{x})+z_0\Psi(\uline{y}))s^2\in \mathfrak{a}+(x_0,z_4,x_4+z_0)$.
\end{enumerate}
\end{lemma}

\begin{proposition}
    In the polynomial ring $
D[x_i,y_i,z_i\mid 0\le i\le4]
$, we have $\alpha^2 s^2 f^2\in (s^4,x_0^4,z_4^4,x_4^4+z_0^4)+\mathfrak{a}$. Hence, $F(\xi)=0$.
\end{proposition}
\begin{proof}
    Simply observe that
{\small
    \begin{align*}
        \alpha^2 s^2 f^2 & \equiv s^4 c^2 \alpha f^3 +x_0^2 R_1 s^2 f^2+z_4^2 R_2 s^2 f^2+ac^3(y_0\Psi(\uline{x})+z_0\Psi(\uline{y}))s^2 f^3 \bmod \mathfrak{a} & \text{by \autoref{lem:poly-identity}\ref{lemma:equiv-1}}\\
        &\equiv ac^3(y_0\Psi(\uline{x})+z_0\Psi(\uline{y}))s^2 f^3\bmod \mathfrak{a}+(s^4,x_0^4,z_4^4)\\
        & \equiv 0\bmod \mathfrak{a}+(s^4,x_0^4,z_4^4,x_4^4+z_0^4)& \text{by \autoref{lem:poly-identity}\ref{lemma:equiv-2}}
    \end{align*}
    }
    using the fact that $f^3\cdot (x_0,z_4,x_4+z_0)A\subseteq(x_0^4,z_4^4,(x_4+z_0)^4)A$.
    \end{proof}

\bibliographystyle{skalpha}
\bibliography{MainBib}

\end{document}